\documentclass[english]{article}
\usepackage[T1]{fontenc}
\usepackage[latin9]{inputenc}
\usepackage[letterpaper]{geometry}
\usepackage[active]{srcltx}
\usepackage{amsthm}
\usepackage{amsmath}
\usepackage{amssymb}
\PassOptionsToPackage{normalem}{ulem}
\usepackage{ulem}
\usepackage[all]{xy}
\usepackage{pb-diagram}

\def\zt{\mathbb{Z}}
\def\R{\mathbb{R}}

\def\F2{\mathbb{F}_2}
\def\zt2{\mathbb{Z}_2}
\def\sq1{\mathrm{Sq}^1}

\makeatletter
\numberwithin{equation}{section}
\numberwithin{figure}{section}
\theoremstyle{plain}
\newtheorem{thm}{\protect\theoremname}[section]
  \theoremstyle{plain}
  \newtheorem{prop}[thm]{\protect\propositionname}
  \theoremstyle{plain}
  \newtheorem{cor}[thm]{\protect\corollaryname}
  \newtheorem{lema}[thm]{Lemma}
   
\@ifundefined{date}{}{\date{}}
\usepackage{tikz-cd}

\makeatother

\usepackage{babel}
  \providecommand{\corollaryname}{Corollary}
  \providecommand{\propositionname}{Proposition}
\providecommand{\theoremname}{Theorem}

\begin{document}
 \title{A Relation Between  Existence of Real Symmetric Nonsingular Bilinear Maps and the  Antisymmetric Index of Projective Spaces.}
 \author{Carlos Dom\'{i}nguez-Albino}
 
 \maketitle

 ~~~~~~~~~~~~~~~~~~~~~~~~~~~~~~~~~~~~~~~~~~~~~~~~~~~~~~~

 \section{Introduction}\label{Sect:int}

 Let $X$ be a topological space, we denote by $F(X,2)$ to $X\times X-\Delta,$ where $\Delta$ is the diagonal in $X\times X,$ of course this is the space of pairs of different points in $X,$ also known as the \textit{configuration space} of two points in $X.$ If now we consider the group of two elements $\mathbb{Z}_2=\{1,-1\}$ there is an action of this group on $F(X,2)$ given by $(-1)(x,y)=(y,x),$ we make reference to this action by the name \textit{symmetric action} of $\mathbb{Z}_2$ on $F(X,2).$ The orbit space coming from the symmetric action is denoted as $B(X,2)$ and is the space of pairs of unordered points in $X$ commonly named the \textit{unordered configuration space} of two points in $X.$

 Let $S^{n-1}$ be the $n-1$ dimensional sphere contained in the Euclidean space $\R^n,$ the antipodal action of $\zt2$ helps to state the classical question, what is the $ \mbox{ minimum }  \{n\in \mathbb{N}\}$ such that there exits an equivariant map, we refer to this kind of map as \textit{antisymmetric,} $$f:F(X,2)\to S^{n-1}$$ with respect to the symmetric and antipodal actions, we name this number \textit{antisymmetric index} of $X$ and it is  denoted by $I_{as}(X).$

 In the direction we are going, this question is important due to its connection with the embedding problem of manifolds stated by A. Haefliger as follows. Suppose $M$ is a $k-$dimensional smooth manifold embedded in $\R^n$ by $g,$ then this would define an antisymmetric map $f:F(M,2)\to S^{n-1}$ by the formula.
 
 $$f(x,y)=\frac{g(x)-g(y)}{||g(x)-g(y)||}$$  
 
 From  the work of Haefliger in \cite{haefliger} S. Feder in \cite{feder} established a partial converse of the previous construction.

 \begin{thm}
 	Suppose $M$ is an $m-$dimensional smooth compact manifold. Then, $M$ embeds in $\R^n$ if there exists an antisymmetric map $F(M,2)\to S^{n-1}$ and $n\geq \frac{3}{2}(m+1).$ 
 \end{thm}

 When $M$ is the real $m-$dimensional projective space $\R P^m$ there are other relations concerning this kind of question coming from a problem in algebra and is about systems of real symmetric bilinear forms.

 Consider a homogeneous system of real symmetric bilinear equations
 
 \begin{eqnarray}
 f_{1}(x,y)=a^1_{11}x_{1}y_{1}+\cdots+a^1_{rr}x_{r}y_{r} & = & 0\nonumber \\
 \vdots & \vdots & \vdots\nonumber \\
 f_{n}(x_{,}y)=a^n_{11}x_{1}y_{1}+\cdots+a^n_{rr}x_{r}y_{r} & = & 0\label{symbilhomogen}
 \end{eqnarray}
 
 It means that every $f_i:\R^r\times \R ^r\to \R, i=1\dots n,$ is  real symmetric ($f_i(x,y)=f_i(y,x)$ for all $x,y\in \R^r$) bilinear form. A trivial solution to this system is one of the form $(0,y)\in \R^r\times \R^r.$ The problem is about whether there  exists a non-trivial solution to (\ref{symbilhomogen}) this problem was established by Stiefel but Hopf in \cite{hopf} gave a topological meaning. In general, the system (\ref{symbilhomogen}) defines a map $$\mu:\R^r\times \R^r\to \R^n$$ which is bilinear and symmetric, when the system only has non-trivial solutions we say that $\mu$ is non-singular.
 
 Hopf's idea is as follows, suppose $\mu$ is bilinear, symmetric, and non-singular then it  defines a map $f:\R P^{r-1}\to S^{n-1}$ by $$f([x])=\frac{\mu(x,x)}{||\mu(x,x)||}.$$
 
 It is not difficult to prove that this map is well defined, continuous, and injective, a topological embedding; and for $r>2$ would give an embedding of $\R P^{r-1}$  in $\R^{n-1}$ . Then the  question, given a projective space $\R P^m$ does there exist an Euclidean model for it? arose in this context. Following this line, real and complex polynomial product provides symmetric non-singular bilinear maps for  odd and even dimensions cases,  giving very well known embedding results for real projective spaces.           
 
 \begin{thm}\label{polemb}
 	Let $r$ be higher than $2.$ If $r$ is odd, $\R P^{r-1}$ is embedded in $\R^{2r-2}.$ For even $r,$ it is embedded in $\R^{2r-3}.$  
  \end{thm}
 
 It is remarkable that Hopf also proved.
 
 \begin{thm}\cite{hopf}\label{nonembHopf}
 	If $r>2$ it is not possible to embed $\R P^{r-1}$ in $\R^r.$
 \end{thm} 
     
     Let us  call $E(k)$ and $N(k)$ to the minimum dimension of a Euclidean space  where you can embed $\R P^k$ and such that there exits a real symmetric non-singular bilinear map from $\R^k\times \R^k$ respectively. Then,  Theorems \ref{polemb} and \ref{nonembHopf} prove at the same time.
     
     \begin{thm}\label{T:embsym}
     	\mbox{}
     	\begin{enumerate}
     		\item $k+2\leq E(k)\leq N(k+1)-1\leq \begin{cases}
     		$2k-1$ \mbox{ if } $k$ \mbox{ is odd,}\\
     		$2k$ \mbox{ if } $k$ \mbox{ is even.}
     		\end{cases}$
     		\item The Fundamental Theorem of Algebra.  
     		\item $N(3)=5,$ \mbox{ and } $N(4)=6.$
     		\item $E(2)=4,$ \mbox{ and } $E(3)=5.$ 
     		
     	\end{enumerate}
     \end{thm} 
     
     And we remark  the following inequalities.  \begin{equation}\label{Eq:asIemb}
      	I_{as}(\R P^k)\leq E(k)\leq N(k+1)-1
      \end{equation}
      
   The previous theorem was proved by H. Hopf using  Alexander duality. N. Steenrod also proved this result \cite{Steenrod} applying steenrod operations. A survey on the topic was given by I. M. James in \cite{James}. From  the integral cohomology of $B(\R P^m,2)$ given in \cite{D8}, in Lemma \ref{alturadeb2} we get the height of an element which from Theorem \ref{T:embsym} and (\ref{Eq:asIemb}) in Section \ref{Sect:Antsyindex} give Theorem \ref{T:asindexPibound}, this  was  previously proved by M. Mahowald \cite{Mahowald} but using Stiefel-Whitney charasteristic classes in combination with steenrod operations, see also \cite{Levine}.  We emphasize in this article the novelty of the proof of Theorem \ref{T:asindexPibound}, which can be considered as a combination of group actions and integral cohomology.         
  
   \section{The Reduced Symmetric Product  of a Projective space and its cohomology}\label{Sect:redsymmprodcoh}
   
  Consider $B(\R P^k,2)$ which is also known as the reduced symmetric product of $\R P^k,$ S. Feder did the calculation of the cohomology ring with $\mathbb{Z}/2\mathbb{Z}$ coeficients of this spaces in \cite{feder}, and now we describe how did he do it. 
  
  It is possible to see $B(\R P^{r-1},2)$ like  pairs of different lines in $\R^r.$ This interpretation gives the fibration \begin{equation}
B(\R P^1,2)\to B(\R P^{r-1},2) \to G_{r,2} 
\end{equation}
where $G_{r,2}$ is the Grassmanian of unoriented two dimensional planes in $\R^{r}.$ The fiber is an open Moebius band and it has  $S^1$ as a deformation retract, this gives a deformation for the total space, which is denoted by $V_{r,2}(D_8),$ and we have a bundle $\eta:$
\begin{equation}\label{FibFeder}
S^1\to V_{r,2}(D_8)\to G_{r,2}
\end{equation}    
       
As Feder pointed out,  every pair of different lines in $\R^r$ define a plane, if we move these lines within that plane until they become mutually orthogonal, then we obtain the same bundle $\eta, $ which now we can   obtain from the canonical two plane bundle $\gamma$ over $G_{r,2}$ when taking its sphere bundle and identifying points which lie on pairs of orthogonal lines. Then we can see $\eta$ like a projectification of a real two dimensional   vector bundle.  

At this point it is important to cite the following very well known theorem which will lead us to some of the main  calculations concerning this work.

\begin{thm} \cite{Bott}\label{cohprojbunSWC}
	Let $\xi$  be a real $k-$dimensional vector bundle over $B,$ and $\R P(\xi)$ its projectification. Then $H^*(\R P(\xi),\mathbb{Z} /2)$ is a free module over $H^*(B,\mathbb{Z}/2)$ generated by $1,X_\xi,\cdots X_\xi^{k-1},$  where $X_\xi\in H^1(\R P(\xi),\mathbb{Z} /2)$ is equal to $w_1(S_\xi),$ $S_\xi$ is the canonical line bundle over $\R P(\xi).$ And there are unique classes $w_i(\xi)\in H^*(B,\mathbb{Z}/2)$ $i=0,\cdots, k, w_0=1,$ such that the equation $$\sum_{i=0}^{k}X_{\xi}^{k-i}w_i(\xi)=0$$ holds in $H^*(\R P(\xi),\mathbb{Z}/2).$ And $w_i(\xi)$ are the Siefel-Whitney classes of the bundle $\xi.$   
\end{thm}

The cohomology of $G_{r,2}$ has the  following  very well known description.

\begin{prop}\cite{feder}
	$$H^*(G_{r,2})\cong \frac{\frac{\mathbb{Z}}{2\mathbb{Z}}[v,w]}{I},$$  such that, $dim(v)=1,$  $dim(w)=2,$ and $I$ is the ideal, of the polynomial ring ${\frac{\mathbb{Z}}{2\mathbb{Z}}[v,w]},$ generated by the two elements: 
	$$\sum_{i=0}^{} \binom{r-1-i}{i}v^{r-1-2i}w^{i}=0 \mbox{ and } \sum_{i=0}^{} \binom{r-i}{i}v^{r-2i}w^{i}=0. $$

The Steenrod algebra is given by $Sq^1(w)=vw.$ 
\end{prop}

From this Proposition, Theorem \ref{cohprojbunSWC}, and the corresponding Stiefel-Whitney classes Feder obtained.

\begin{thm}\cite{feder}\label{T:cohunconfz2}
	$$H^*(B(\R P^{r-1},2),\mathbb{Z}/2\mathbb{Z})\cong \frac{\frac{\mathbb{Z}}{2\mathbb{Z}}[u,v,w]}{J},$$  such that, $dim(u)=dim(v)=1,$  $dim(w)=2,$ and $J$ is the ideal, of the polynomial ring ${\frac{\mathbb{Z}}{2\mathbb{Z}}[u,v,w]},$ generated by the three elements: 
	$$\sum_{i=0}^{} \binom{r-1-i}{i}v^{r-1-2i}w^{i} \mbox{, } \sum_{i=0}^{} \binom{r-i}{i}v^{r-2i}w^{i}, \mbox{ and }uv+u^2.$$	
\end{thm}

 In this ring is fulfilled the following   algebraic fact, compare this result to Lemma \ref{alturadeb2}.
      
      \begin{cor}\label{Cor:heightv}
      	Let $r\geq 4,$ $r=2^k+s, 1\leq s\leq 2^k,$ and consider the element $v\in H^*(B(P^{r-1},2);\mathbb{Z}/2\mathbb{Z}).$ Then, $$height(v_1)=2^{k+1}-1$$ 
      \end{cor} 
       
\section{Integral Cohomology of the Reduced Symmetric Product and Antisymmetric Index of Real Projective Spaces of dimensions $5$ and $6.$}\label{Sect:Antsyindex}\mbox{}

This section is about the calculation  of  $I_{as}(\R P^{2^m+1)}$ in comparison to $E(\R P^{2^m+1)}),$ using the integral cohomology of $B(\R P^k,2),$  which was obtained in \cite{D8}  from the previous knowledge of $\mathbb{Z}/2\mathbb{Z}$ cohomology described in \ref{T:cohunconfz2} and the application of the Bokstein spectral sequence. The rings obtained are as follows.

  \begin{thm}\cite{D8}\label{chb2pmpar}
  	Let $m=2t$, $t\geq1$. The integral cohomology ring $H^*(B(\R P^m, 2))$ is generated by five classes $a_2, b_2, c_3, d_4, e_{2m-1},$ subscripts denote dimension of the corresponding element, subject only to the relations (where we are omitting  the subscripts):
  	\begin{enumerate}
  		\item $2a = 2b = 2c = 4d = 0;$
  		\item $b^2 = ab;$
  		\item $c^2 = ad;$
  		\item $\sum\binom{i+j}{j}a^icd^j=0,\,$
  		where the sum runs over $\,i,j\geq 0\,$ with
  		$\,i+2j=t-1;$
  		
  		\vspace{-2.5mm}
  		\item $\sum\binom{i+j}{j}a^ibd^j=
  		\begin{cases}2d^{\frac{t+1}2},&\!\!\!t\;\mbox{odd},\\0,&
  		\!\!\!t\;\mbox{even},\end{cases}\,$ where the sum runs
  		over $\,i,j\geq 0\,$ with $\,i+2j=t;$
  		\vspace{-2.2mm}
  		\item $\sum\binom{i+j}{j}a^{i+1}d^j=0,\,$ where the sum runs
  		over $\,i,j\geq 0\,$ with $\,i+2j=t;$
  		\item $\sum\binom{i+j}{j}a^icd^j=0,\,$ where the sum runs
  		over $\,i,j\geq 0\,$ with $\,i+2j=t;$
  		
  		\vspace{-2.5mm}
  		\item $\sum\binom{i+j}{j}a^ibd^{j+1}=\begin{cases}
  		2d^{\frac{t+2}2},&\!\!\!t\;\mbox{even},\\0,&
  		\!\!\!t\;\mbox{odd},\end{cases}\,$ where the sum runs
  		over $\,i,j\geq 0\,$ with $\,i+2j=t-1;$
  		
  		\vspace{-2.5mm}
  		\item $d^t=0;$
  		\item $e\varepsilon= 0,\,$ for $\varepsilon\in\{a,b,c,d,e\}$.
  	\end{enumerate}
  \end{thm}

\begin{thm}\cite{D8}\label{chb2pmimpar}
	Let $m=2t+1$, $t\geq0$. The integral cohomology ring $H^*(B(P^m, 2))$ is generated by five classes $a_2, b_2, c_3, d_4, e_m,$  subscripts denote dimension of the corresponding element and we omit them from now on, subject only to the  relations:
	\begin{enumerate}
		\item $2a = 2b = 2c = 4d = 0;$
		\item $b^2 = ab;$
		\item $c^2 = ad;$
		
		\vspace{-2.5mm}
		\item $\sum\binom{i+j}{j}a^ibd^j=
		\begin{cases}2d^{\frac{t+1}2},&\!\!\!t\;\mbox{odd},\\0,&
		\!\!\!t\;\mbox{even},\end{cases}\,$ where the sum runs
		over $\,i,j\geq 0\,$ with $\,i+2j=t;$
		
		\vspace{-2.2mm}
		\item $\sum\binom{i+j}{j}a^{i+1}d^j=0,\,$ where the sum runs
		over $\,i,j\geq 0\,$ with $\,i+2j=t;$
		\item $\sum\binom{i+j}{j}a^icd^j=0,\,$ where the sum runs
		over $\,i,j\geq 0\,$ with $\,i+2j=t;$
		
		\vspace{-2.5mm}
		\item $\sum\binom{i+j}{j}a^ibd^j=\begin{cases}
		2d^{\frac{t+2}2},&\!\!\!t\;\mbox{even},\\0,&
		\!\!\!t\;\mbox{odd},\end{cases}\,$ where the sum runs
		over $\,i,j\geq 0\,$ with $\,i+2j=t+1;$
		
		\vspace{-2.2mm}
		\item $\sum\binom{i+j}{j}a^{i+1}d^j=0,\,$ where the sum runs
		over $\,i,j\geq 0\,$ with $\,i+2j=t+1;$
		\item $\sum\binom{i+j}{j}a^icd^j=0,\,$ where the sum runs
		over $\,i,j\geq 0\,$ with $\,i+2j=t+1;$
		\item $d^{t+1}=0;$
		\item
		\begin{enumerate}
			\item $e^2=0,$
			\item $\mu e=\kappa b^{\kappa}cd^l,$
			\item $ce=\eta d^{l+1},$
			\item $\mbox{ and } de=\sum_{i=1}^l\binom{t-i}{i-1}a^{t-2i}bcd^i$.
			
			Here $\mu\in\{a,b\},t=2l+\kappa $  with $\kappa\in\{0,1\}$ and $\eta=b,$ if $\kappa=1$ whereas $\eta=2$ if
			$\kappa = 0$, except perhaps for $m = 5.$
		\end{enumerate}
		
	\end{enumerate}
\end{thm}

 Note that this description is presented in a little bit more explicit way than \cite{D8}, the reason is the proof of next lemma where we use this algebraic structures to obtain the height of a certain relevant element.

\begin{lema}\label{alturadeb2}
	Consider the element $b\in H^2(B(P^m,2);\mathbb{Z})$ coming from  previous theorems. 	If $k$ is the smallest positive integer such that  $b^k=0:$ 
	
	Suppose $m=2^e.$ Then
  $$k=2^e.$$
	
	On the other hand, let $m\in \{2^e+1,\ldots,2^{e+1}-1\}.$  Then  $$k=2^e+1.$$
	
\end{lema}

\begin{proof}
	First suppose $m=2^e+1;$ from 4) and 7) in Theorem \ref{chb2pmimpar} we have  \begin{equation}\label{parahb1}a^tb=\binom{t-1}{1}a^{t-2}bd+\binom{t-2}{2}a^{t-4}bd^2+\cdots\end{equation}
	and \begin{equation}\label{parahb2}a^{t-1}bd=\binom{t-2}{1}a^{t-3}bd^2+\binom{t-3}{2}a^{t-5}bd^3+\cdots+2d^{\frac{t+2}{2}}.\end{equation} Note that, in this case, $t=2^{e-1},$ then using the formula $$\binom{2^{e-1}-k}{k-1}\equiv 0 \mod 2\;\forall\; k$$ it follows that
	\begin{equation}\label{eqalt}
	\mbox{$a^{t-1}bd=2d^{\frac{t+2}{2}}$ and therefore
		$a^tbd=0.$}
	\end{equation}
	
	Multiplying  (\ref{parahb1}) by $a^{t-1}$ and using (\ref{eqalt}) we  get
	
	\begin{equation}\label{hb22e}
	b^{2^e}=b^{m-1}=b^{2t}=a^{2t-1}b=a^{t-1}bd^{t/2}=2d^{\frac{2t+2}{2}}d^{t/2-1}=2d^{t}\neq 0.
	\end{equation}
	
	For dimensional conditions in Theorem \ref{chb2pmimpar}, it is obvious that $b^{2^{e}+1}=0,$ therefore the proof is done for this case. Note that this is all what we need to prove the main result of this work.

	If we consider the case $m=2^{e+1}-1,$ then $t=2^e-1,$ and from 4  in \ref{chb2pmimpar}  we get
	\begin{equation}\label{parahb11}a^tb=\binom{t-1}{1}a^{t-2}bd+\binom{t-2}{2}a^{t-4}bd^2+\cdots + 2d^{\frac{t+1}{2}}\end{equation} 
	and from here applying 
	$$\binom{2^{e}-k}{k-1}\equiv 0 \mod 2\;\forall\; k $$ we get
	\begin{equation} \label{hb22e1}
b^{2^e}=	b^{t+1}=a^{t}b=2d^{\frac{t+1}{2}}\neq0 \mbox{ and }b^{2^e+1}=b^{t+2}=a^{t+1}b=aa^tb=2ad^{\frac{t+1}{2}}=0.
	\end{equation}
	
	It is not difficult to analyze the case $m=2^e$ in a similar manner, but this time using \ref{chb2pmpar}, the rest of the cases now follow from these three cases and an inductive argument.
	
\end{proof}

Now, suppose there exist an antisymmetric map $f:F(\R P^{r-1},2) \to S^{n-1},$ passing to the orbit spaces we have the following diagram.

\begin{equation}\label{Diag:conf}
\xymatrix{F(\R P^{r-1},2)\ar[r]^f \ar[d] & S^{n-1}\ar[d]\\
	B(P^{r-1},2)\ar[r]^{P(f)}&P^{n-1}}
\end{equation}

Then the generator $z$ in the second cohomology group of $\R P^{n-1}$ corresponds to the element $b_2$ appearing in \ref{alturadeb2} under $P(f)^*,$ this fact is proved in \cite{D8} using group actions on the Stiefel manifold $V_{r,2},$ as is described now.

 Consider the dihedral group of order $8$
 
 \begin{equation} \label{Dih}
 D_8=\{t,y|t^4=y^2=1,yt=t^3y\}
 \end{equation}
 
 $D_8$ acts freely on the Stiefel manifold $V_{r,2}\subset \R^r\times\R^r$ as $t(v_1,v_2)=(v_2,-v_1), y(v_1,v_2)=(v_1,-v_2).$ 
 
 Let  $H=<y,yt^2>\subset D_8$ denote the subgroup of $D_8$ generated by $y,yt^2,$ it is not difficult to prove that $H\cong \zt2 \times \zt2.$ If we restrict the action of $D_8$ to $H$ on $V_{r,2},$ then the orbit space $\frac{V_{r,2}}{H},$ has the same homotopy type than $F(\R P^{r-1},2)$ and same thing for $\frac{V_{r,2}}{D_8}$ in comparison to $B(\R P^{r-1},2),$  $B(\R P^{r-1},2)$ has the same homotopy type than $\frac{V_{r,2}}{D_8}$:   To prove it for $ B(\R P^{r-1}),$ just consider the map $g:V_{r,2}\to B(\R P^{r-1},2)$
 given by $g(v_1,v_2)=([v_1],[v_2])$ this map clearly pass to the quotient and gives a map $\frac{V_{r,2}}{D_8}\to B(\R P^{r-1},2).$ The homotopy inverse is provided by Gram-Schmidt orthogonalization  process, applied to generators of a pair of different lines in $\R^{r},$ everything is well defined and works right due to identifications on the orbit space, and the argument is similar for $F(\R P^{r-1},2),$ 
 compare to the bundle (\ref{FibFeder}) in Section  \ref{Sect:redsymmprodcoh}. In particular this gives a homotopy type fibration. \begin{equation}\label{Eq:fibresimp}
  V_{r,2}\to B(\R P^{r-1},2)\to BD_8
 \end{equation} 
 where $BD_8$ is the classifying space of $D_8.$ 
 
 Now the exact sequence of groups $$1\to \zt2 \times \zt2\cong H \to D_8\to \frac{D_8}{H}\cong \zt2\to 1,$$    (\ref{Eq:fibresimp}),  and the double covering fibration $$F(\R P^{r-1})\to B(\R P,2) \to \R P^\infty$$ where used to  prove that, related to Diagram (\ref{Diag:conf}) $P(f)^*(z)=b.$ Then a contradiction argument and \ref{alturadeb2}, prove the following. 

\begin{thm}\label{T:asindexPibound}
	$$2^{m+1}< I_{as}(\R P^{2^m+1})$$
\end{thm}  

\begin{proof}
	Suppose there exists an antisymmetric map $f:F(\R P^{2^m+1},2)\to \R P^{2^{m+1}-1},$ then from previous paragraph $P(f)^*(z)=b$ but $z^{2^m}=0$ while from \ref{alturadeb2} $b^{2^m}\neq 0.$
\end{proof}

E. Rees proved in \cite{rees}  that $I_{as}(\R P^6)\leq 9,$ due to the fact $I_{as}(\R P^k)\leq I_{as}(\R P^{k+1})$ for all $k,$ we have.

\begin{cor}
	$$I_{as}(\R P^5)=I_{as}(\R P^6)=9.$$
\end{cor}

The previous Corollary  was obtained   in the  PhD.  Author's Thesis, and was published  first in \cite{D8}. But in the following theorem, we  recover a more  general result in connection to the embedding  problem of projective spaces and real symmetric bilinear maps.   

\begin{thm}\label{C:iboundem}\mbox{}
	
	\begin{enumerate}
		\item $E(2^m+1)= I_{as}(\R P^{2^m+1})=2^{m+1}+1.$
		\item $N(2^m+2)=2^{m+1}+2.$
		
	\end{enumerate}
\end{thm}
 
 \begin{proof}
 	Theorem \ref{T:asindexPibound} gives $2^{m+1}+1\leq I_{as}(\R P^{2^m+1})$, join this inequality to (\ref{Eq:asIemb}) and Theorem \ref{T:embsym} to obtain $$2^{m+1}+1\leq I_{as}(\R P^{2^m+1})\leq E(2^m+1)\leq N(2^m+2)-1\leq 2^{m+1}+1.$$ 
 \end{proof}

  {\scriptsize CENTRO DE CIENCIAS MATEM\'ATICAS,UNIVERSIDAD NACIONAL AUT\'ONOMA DE M\'EXICO, CAMPUS MORELIA, APARTADO POSTAL 61-3 (XANGARI), MORELIA, MICHOAC\'AN, MEXICO 58089
  
  E-mail address:cda@matmor.unam.mx.
 
 Supported from DGAPA-UNAM postdoctoral scholarship.}

\end{document}